\documentclass[3p,times,chicago]{elsarticle}

\usepackage{amssymb}

\usepackage{amsmath,amssymb,amsthm,textcomp}
\usepackage{amsfonts,graphicx}
\usepackage[mathscr]{eucal}
\usepackage{color}
\usepackage{hyperref}
\usepackage[table]{xcolor}
\usepackage{booktabs}

\usepackage[]{natbib}
\usepackage{subfig}
\usepackage{epsfig}
\usepackage{caption}
\theoremstyle{definition}
\usepackage{float}
\usepackage{mathtools}  
\usepackage{diffcoeff}  
\floatstyle{plaintop}
\usepackage[figuresright]{rotating}
\usepackage{natbib}
\usepackage{amssymb,amsmath, amsfonts}
\usepackage{enumerate}
\usepackage{amsthm}
\usepackage[utf8]{inputenc}
\usepackage[english]{babel}
\setcitestyle{authoryear,open={(},close={)}} 

\usepackage{graphicx}
\usepackage{caption}
\newtheorem{theorem}{Theorem}[section]

\newtheorem{remark}{Remark}[section]
\newtheorem{lemma}{Lemma}[section]
\newtheorem{definition}{Definition}[section]

\numberwithin{equation}{section}

\newtheoremstyle
{remarkstyle}
{}
{11pt}
{}
{}
{\bfseries}
{:}
{     }
{\thmname{#1} \thmnumber{#2} }

\theoremstyle{remarkstyle}

\makeatletter
\def\ps@pprintTitle{%
	\let\@oddhead\@empty
	\let\@evenhead\@empty
	\let\@oddfoot\@empty
	\let\@evenfoot\@oddfoot
}
\makeatother
\begin{document}

\begin{frontmatter}

\title{Tempered space-time fractional negative binomial process}
%

\author[add1]{Shilpa}
\ead{shilpa17garg@gmail.com}
\author[add1]{Ashok Kumar Pathak\corref{cor1}}
\ead{ashokiitb09@gmail.com}
\author[add2]{Aditya Maheshwari}
\ead{adityam@iimidr.ac.in}
\address[add1]{Department of Mathematics and Statistics, Central University of Punjab, Bathinda, India}
\address[add2]{Operations Management and Quantitative Techniques Area, Indian Institute of Management Indore, Indore, India}
\cortext[cor1]{Corresponding author}
\begin{abstract} 
In this paper, we define a tempered space-time fractional negative binomial process (TSTFNBP)  by subordinating the fractional Poisson process with an independent tempered Mittag-Leffler L\'{e}vy subordinator. 
We study its
distributional properties and its connection to
partial differential equations. We derive the asymptotic behavior of its fractional order moments and long-range dependence property. It is shown that the TSTFNBP exhibits overdispersion. We also obtain some results related to the first-passage time. 
\end{abstract}
\begin{keyword} 
Fractional Poisson process, Tempered Mittag-Leffler subordinator,  Fractional moments, LRD, PDEs.

2020 MSC: Primary 60G22, 60G51, Secondary 60G55, 60E05

\end{keyword}

\end{frontmatter}



\section{Introduction}
The applications of subordinated count processes in financial mathematics and actuarial sciences (see \cite{appl-act1,appl-fin2,appl-fin3}) have created an impetus for sustained enquiry in their theoretical research. In the last two decades, researchers have studied the fractional Poisson processes (FPPs) from a stochastic subordination point of view. In this direction,  we find several varieties of the  subordinated FPP being studied (see \cite{meerschaert2011fractional, orsingher2012space, maheshwari2019fractional, gupta2023fractional, meoli2023some,soni2024generalized}). It has  led to the development of rich literature in the domain. Several other practical applications of these processes can be found in various disciplines such as economics, finance, actuarial science,  physics,  infectious diseases modeling, and reliability (see \cite{doukhan2002theory, biard2014fractional, guler2022forecasting, di2023competing, Ritik2024}). 

Following the subordination approach, we narrow our attention to the negative binomial (NB) process due to its advantages for modelling overdispersed data. The NB process is a time-changed Poisson process delayed by an independent gamma subordinator. Recently, several fractional variants of the NB process have been explored; for example, \cite{samy2018fractional} introduced the fractional NB process (FNBP) by subordinating the FPP with an independent gamma subordinator and studied its governing partial differential equations (PDEs),  \cite{beghin2018space} considered a space fractional NB process (SFNBP) and used it in biological modeling. Moreover, several tempered stable extensions of the NB process have been proposed and studied in the literature. The tempered stable subordinator is obtained by exponential tempering of the stable process, which exhibits heavy tail behavior of a stable process at short times and lighter tails at large times,  additionally making all its moments finite. \cite{maheshwari2023tempered} defined a tempered space fractional negative binomial process (TSFNBP) and explored its distributional and long-range dependence (LRD) properties. Recently, \cite{Ritik2024} discussed applications of the tempered space fractional Poisson process to the reliability and bivariate shock models. This paper will consider a tempered space-time fractional negative binomial process using the stochastic subordination approach. In the following paragraph, we provide a brief description of the construction of the process and its important properties.

The tempered Mittag-Leffler L\'{e}vy subordinator (see \cite{kumar2019tempered}) is defined as a tempered stable subordinated model delayed by an independent gamma subordinator. Its distribution is semi-heavy tailed, and it is an important feature for studying extreme phenomena. As a result, it can be used in place of the gamma subordinator to construct several time-varying stochastic processes. 
In this paper, we introduce the FPP time-changed by the tempered Mittag-Leffler L\'evy process. We discuss several important characteristics of the process and derive various asymptotic results for the distributional properties and dependence structure.  In particular, we derive the probability mass function (pmf) and discuss its connections with PDEs. We derive the asymptotic behavior of its fractional order moments and examine its dependence properties. Some results related to the first-passage time are also explored. 

The structure of the article is as follows. In Section \ref{sec:prelims}, we present some preliminary notations, definitions, and results. In Section \ref{sec:tstfnbp}, we define the TSTFNBP and discuss its main distributional characteristics, and we derive the asymptotic behavior of its fractional order moments. In Section \ref{sec:lrd}, we study its dependence properties and some results related to the first-passage time.

\section{Preliminaries}\label{sec:prelims} This section introduces some notations, definitions, elementary distributions, and results that will be used in the following sections. Let $\mathbb{N}$, $\mathbb{R}$, and $\mathbb{C}$ denote the set of all natural, real, and complex numbers, respectively. Let $\mathbb{Z}_{+}=\mathbb{N}\cup \{0\}$ denotes the set of all non-negative integers.
\subsection{Definitions, some elementary distributions and results}\label{subsec:1}
\noindent (i)	Let $f :[a,b] \subset \mathbb{R}\longrightarrow\mathbb{R}$ be such that $f(t)$ is $(n+1)$ times continuous differentiable for $n < \tau <n+1$. Then, the Riemann-Liouville fractional derivative of order $\tau>0$ is defined as (see \cite{podlubny1999introduction})
\begin{equation*}
_aD^\tau_tf(t)=\bigg( \frac{d}{dt}\bigg)^{n+1}\int_{a}^{t}(t-u)^{n-\tau}f(u)du.
\end{equation*}
\noindent (ii)The generalized Wright function is defined by (\cite{kilbas2002generalized})
\begin{equation}\label{gwf11}
_p\psi_q \left[z\; \vline \;\begin{matrix}
\left(\alpha_i, \beta_i\right)_{1,p}\\
(a_j,b_j)_{1,q}
\end{matrix} \right] = \sum_{k=0}^{\infty} \frac{z^k}{k!} \frac{\prod_{i=1}^{p} \Gamma(\alpha_i + \beta_i k)}{\prod_{j=1}^{q}\Gamma(a_j + b_j k)},\;\;  z, \alpha_i, a_i \in \mathbb{C}\; \text{and}\; \beta_i, b_i \in \mathbb{R},
\end{equation}
\noindent (iii) For  $ 0 < \beta < 1$, let  $\{N_\beta(t, \lambda)\}_{t\geq 0}$ be a FPP having parameter $\lambda>0$. Its one-dimensional distributions are (see \cite{laskin2003fractional, meerschaert2011fractional})
\begin{equation*}
p_{\beta}(n/t, \lambda) =P[N_\beta(t, \lambda)=n] =\frac{(\lambda t^{\beta})^n}{n!} \sum_{k=0}^{\infty} \frac{(n+k)!}{k!}\frac{(-\lambda t^{\beta})^k}{\Gamma(\beta(k+n)+1)},\;\; n\in \mathbb{Z}_{+}.
\end{equation*}
\noindent (iv) Let $\Upsilon(t)\sim G(\lambda_1, \beta_1 t)$. Its probability density function (pdf) is given by 
\begin{equation*}
f_G(x,t) = \frac{\lambda_1^{\beta_1 t}}{\Gamma(\beta_1 t)}x^{\beta_1 t-1}e^{-\lambda_1 x}, \;\; x>0.
\end{equation*}
\noindent (v) For $\alpha\in (0,1)$ and $ \mu>0$, let $S_{\alpha,\mu}(t)$ be a tempered $\alpha$-stable subordinator (TSS). Then its pdf $g_{\alpha,\mu}(x,t)$  is given by (see \cite{rosinski2007tempering})
\begin{equation*}
	g_{\alpha,\mu}(x,t) = e^{-\mu x+\mu^{\alpha}t}g_{\alpha}(x,t),
\end{equation*}
where $g_{\alpha}(x,t)$ is the pdf of $\alpha$-stable subordinator (see \cite{kumar2015inverse}).\\
\noindent (vi) The tempered Mittag-Leffler L\'{e}vy process (TMLLP) $M_{\alpha,\beta_1,\lambda_1,\mu}(t)$  is obtained by subordinating TSS with an independent gamma subordinator as $M_{\alpha,\beta_1,\lambda_1,\mu}(t):= S_{\alpha,\mu}(G_{\lambda_1,\beta_1}(t)),\; \alpha\in(0,1),\lambda_1,\mu,\beta_1 >0, t\geq 0 $. Its pdf $f_{M_{\alpha,\beta_1,\lambda_1,\mu}(t)}$ is given by (see \cite{kumar2019tempered}) 

\begin{equation}\label{pmftm}
	f_{M_{\alpha,\beta_1,\lambda_1,\mu}}(t,x)={\lambda_1}^{\beta_1t}{e}^{-\mu x}\sum_{k=0}^{\infty } \frac{(-1)^k(\lambda_1-\mu^{\alpha})^k\Gamma(\beta_1 t+k) x^{\alpha(\beta_1t+k)-1}}{\Gamma (k+1)\Gamma(\beta_1 t)\Gamma(\alpha(\beta_1t+k) )},\hspace{0.3cm} \lambda_1> \mu^{\alpha },x> 0
\end{equation}
The L\'{e}vy measure density $\pi(x)$ and the Laplace transform (LT) of the TMLLP are respectively
\begin{equation}\label{ld1}
	\pi(x)= \frac{\alpha\beta_1}{x}e^{-\mu x}E_{\alpha, 1}\left[(\mu^{\alpha}-\lambda_1) x^\alpha\right],\; \lambda_1> \mu^{\alpha }, x> 0,
\end{equation}
where two parameters Mittag-Leffler function $E_{\alpha, \beta}(z)$ is defined as (see \cite{podlubny1999introduction})
\begin{equation}
E_{\alpha, \beta}(z) = \sum_{k=0}^{\infty} \frac{z^k}{\Gamma(\alpha k+\beta)}, \;\; \alpha, \beta> 0 
\end{equation}
and	
\begin{equation}\label{lt1}
		\mathbb{E}\left[e^{-uM_{\alpha,\beta_1,\lambda_1,\mu}(t)}\right] =\left(\frac{\lambda_1}{\lambda_1-\mu^{\alpha}+(\mu+u)^{\alpha}}\right)^{\beta_1t}.
\end{equation}

\section{Tempered space-time fractional negative binomial process}\label{sec:tstfnbp}

In this section, we define the TSTFNBP process and derive its distributional properties
\begin{definition}
    
Let $\{N_{\beta}(t,\lambda)\}_{t\geq 0}$ be an FPP with parameter $\lambda>0$. The tempered space-time fractional negative binomial process (TSTFNBP), denoted by $\{\mathcal{Q}_{\alpha ,\beta _1,\mu }^{\lambda _1,\beta }(t,\lambda)\}_{t\geq 0}$, which is  obtained by the subordination of the FPP with an independent TMLLP (\ref{subsec:1} (vi)), that is,
 \begin{equation*} 
 \mathcal{Q}_{\alpha ,\beta _1,\mu }^{\lambda _1,\beta }(t,\lambda):=	N_\beta(M_{\alpha,\beta_1,\lambda_1,\mu}(t),\lambda),\;\; t\geq0.
 \end{equation*}
\end{definition}

Let $\lambda_1>\mu^{\alpha}$ and $y>0$, the pmf of  $\{\mathcal{Q}_{\alpha ,\beta _1,\mu }^{\lambda _1,\beta }(t,\lambda)\}_{t \geq 0}$, denoted by $p_{\alpha ,\beta _1,\mu }^{\lambda _1,\beta }(n,t)=\mathbb{P}(\mathcal{Q}_{\alpha ,\beta _1,\mu }^{\lambda _1,\beta }(t,\lambda)=n)$ is derived as 
\begin{align*}
p_{\alpha ,\beta _1,\mu }^{\lambda _1,\beta }(n,t)
&=\int_{0}^{\infty }p_{\beta }\left({n}/{y},\lambda \right) f_{M_{\alpha,\beta_1,\lambda_1,\mu}}(t,y)dy\\
&=\int_{0}^{\infty}\left (\frac{(\lambda y^{\beta })^n}{n!}\sum_{k=0}^{\infty}\frac{(n+k)!}{k!} \frac{(-\lambda y^{\beta})^{k}}{\Gamma (\beta (n+k)+1)} \right )f_{M_{\alpha,\beta_1,\lambda_1,\mu}}(t,y)dy\\
&= \frac{\lambda ^n}{n!}\sum_{k=0}^{\infty}\frac{(n+k)!}{k!} \frac{(-\lambda)^ {k}}{\Gamma (\beta (n+k)+1)} \int_{0}^{\infty} y^{\beta (n+k) } f_{M_{\alpha,\beta_1,\lambda_1,\mu}}(t,y)  dy\\
&=\frac{\lambda ^n}{n!}\sum_{k=0}^{\infty}\frac{(n+k)!}{k!} \frac{(-\lambda)^ {k}}{\Gamma (\beta (n+k)+1)} \mathbb{E}\left[({M_{\alpha,\beta_1,\lambda_1,\mu}(t) })^{\beta(n+k)}\right],
\end{align*}
as $\mathbb{E}\left[({M_{\alpha,\beta_1,\lambda_1,\mu}(t) })^{\;\rho}\right] < \infty$ for all $\rho > 0$ (see \cite{kumar2019tempered}).
\begin{remark} \noindent(i) When $\alpha =1,\mu=0$, the pmf of TSTFNBP reduces to  
		\begin{equation*}
		p_{1 ,\beta _1,0 }^{\lambda _1,\beta }(n,t)=\frac{\lambda^n}{\lambda_1^{\beta n} n! \Gamma(\beta_1 t)}\;  _2\psi_1 \left[\frac{-\lambda}{\lambda_1^{\beta}}\; \vline \;\begin{matrix}
\left(n+1,1\right),& \left(\beta_1 t + \beta n, \beta\right)\\
(1+\beta n, \beta)
\end{matrix} \right],
		\end{equation*}
		which is the pmf of the FNBP discussed in \cite{samy2018fractional}. In addition, for $\beta=1$, it leads to  the pmf of the NB$(\beta_1 t, \frac{\lambda}{\lambda_1+\lambda})$ of the form
		\begin{equation*}
		p_{1 ,\beta _1,0 }^{\lambda _1,1 }(n,t)=\binom{n+\beta _{1}t-1}{n}\left ( \frac{\lambda _1}{\lambda _1+\lambda } \right )^{\beta _{1}t}\left ( \frac{\lambda }{\lambda _1+\lambda } \right )^{n} ,
		\end{equation*}
		as discussed in \cite{samy2018fractional} .\\
   \noindent(ii) When $\mu=0$, the TSTFNBP corresponds to the generalized fractional negative binomial process  defined in \cite{soni2024generalized}.
  \end{remark}
  Next, we obtained the governing fractional PDE for the pmf of TSTFNBP  $\{\mathcal{Q}_{\alpha ,\beta _1,\mu }^{\lambda _1,\beta }(t,\lambda)\}_{t \geq 0}$. First, we proceed with the following lemma. 
\begin{lemma}\label{lemm1}
	(\cite{samy2018fractional}) For any $\tau \geq 1$, the governing fractional PDE of order $\tau$ for the gamma subordinator $\{\Upsilon(t)\}_{t\geq 0}$ is given by
	\begin{equation*}
	\diffp{^\tau}{t^\tau}	f_G(x,t) = \beta_1	\diffp{^{\tau-1}}{t^{\tau-1}}\left[\log \lambda_1 +\log x - \psi(\beta_1 
	t)\right]	f_G(x,t), \;\; x >0 \text{ and }
	f_G(x,0) =0,
	\end{equation*}
	where $\psi(x)$ is the digamma function and $	\diffp{^\tau}{t^\tau}(\cdot)$ is the Riemann-Liouville fractional differential operator.
\end{lemma}
\noindent The next theorem gives the PDE with respect to time variable satisfying the pdf of the TMLLP.
\begin{theorem}\label{prop1}
	Let $g_{\alpha,\mu}(x,t)$ be the pdf of the TSS. Then the pdf of TMLLP satisfies the following fractional PDE
	\begin{equation*}
	\diffp{^\tau}{t^\tau}  f_{M_{\alpha,\beta_1,\lambda_1,\mu}}(x,t)=\beta_1	\diffp{^{\tau-1}}{t^{\tau-1}}\left [ (\log\lambda_1-\psi (\beta_1t))f_{M_{\alpha,\beta_1,\lambda_1,\mu}}(x,t)+\int_{0}^{\infty } g_{\alpha,\mu}(x,y)(\log y)f_{G}(y,t)dy\right ],x>0,t>0,
	\end{equation*}
 with $f_{M_{\alpha,\beta_1,\lambda_1,\mu}}(x,0)=0.$
\end{theorem}
\begin{proof} Consider
	\begin{equation*}
f_{M_{\alpha,\beta_1,\lambda_1,\mu}}(x,t)=\int_{0}^{\infty } g_{\alpha,\mu}(x,y)f_{G}(y,t)dy.
	\end{equation*}
Using the Riemann-Liouville fractional derivative, we obtain
	\begin{align*}
	\diffp{^\tau}{t^\tau} f_{M_{\alpha,\beta_1,\lambda_1,\mu}}(x,t)
	&= \diffp{^\tau}{t^\tau}\int_{0}^{\infty }g_{\alpha,\mu}(x,y)f_{G}(y,t)dy\\
	&= \int_{0}^{\infty }g_{\alpha,\mu}(x,y)\diffp{^\tau}{t^\tau} 	f_G(y,t) dy\\
	&= \int_{0}^{\infty} g_{\alpha,\mu}(x,y) \left[\beta_1 \diffp{^{\tau-1}}{t^{\tau-1}}\left[\log \lambda_1 +\log y - \psi(\beta_1 t)\right]	f_G(y,t)\right]dy \;\;\;(\text{using  Lemma \ref{lemm1}})\\
	&= \beta_1	\diffp{^{\tau-1}}{t^{\tau-1}}\int_{0}^{\infty} g_{\alpha,\mu}(x,y)\left(\log \lambda_1 - \psi(\beta_1 t) \right) f_G(y,t)dy+\beta_1 \int_{0}^{\infty}g_{\alpha,\mu}(x,y) (\log y) 	\diffp{^{\tau-1}}{t^{\tau-1}}f_G(y,t)dy.\qedhere
	\end{align*}
\end{proof}
With the help of Theorem \ref{prop1}, we can now obtain the following result.
\begin{theorem} For  $\tau \geq 1$, the pdf  $p_{\alpha,\beta _1,\mu }^{\lambda _1,\beta }(n,t)$ of TSTFNBP satisfies
\begin{align*}
&\frac{1}{\beta_1}	\diffp{^\tau}{t^\tau} 	p_{\alpha ,\beta _1,\mu }^{\lambda _1,\beta }(n,t) = 	\diffp{^{\tau-1}}{t^{\tau-1}}\left[\left(\log \lambda_1 - \psi(\beta_1 t) \right) 	p_{\alpha ,\beta _1,\mu }^{\lambda _1,\beta }(n,t)+\int_{0}^{\infty}\int_{0}^{\infty}	p_{\beta}(n/y_1, \lambda)g_{\alpha,\mu}(y_1,y) (\log y) f_G(y,t)dy dy_1\right],
\end{align*}
\text{with}\; $p_{\alpha ,\beta _1,\mu }^{\lambda _1,\beta }(0,0) = 1.$
\end{theorem}
\begin{remark}
\noindent When $\mu=0,$ the governing PDE of the TSTFNBP reduces to \begin{equation*}
\frac{1}{\beta_1}	\diffp{^\tau}{t^\tau} 	p_{\alpha ,\beta _1,0 }^{\lambda _1,\beta }(n,t) = 	\diffp{^{\tau-1}}{t^{\tau-1}}\left[\left(\log \lambda_1 - \psi(\beta_1 t) \right) 	p_{\alpha ,\beta _1,0 }^{\lambda _1,\beta }(n,t)+\int_{0}^{\infty}\int_{0}^{\infty}	p_{\beta}(n/y_1, \lambda)g_{\alpha}(y_1,y) (\log y) f_G(y,t)dy dy_1\right], \text{with}\; p_{\alpha ,\beta _1,\mu }^{\lambda _1,\beta }(0,0) = 1,
\end{equation*} as reported by \cite{soni2024generalized}. Moreover,
    when $\beta=1$, it 
corresponds to the PDE of the SFNB as studied in \cite{beghin2018space}.
\end{remark}
Next, we discuss the asymptotic behavior for the moments for the TMLLP.
\begin{theorem}\label{prop2} 
	Let $q>0$, the asymptotic behavior of $q^{th}$ order moments of TMLLP is given by
 \begin{equation*}\mathbb{E}({M_{\alpha,\beta_1,\lambda_1,\mu}(t) })^{q}\sim \left ( \frac{\alpha \beta_1\mu ^{\alpha -1}}{\lambda_1}  \right )^qt^q \;\text{as} \;  t \to \infty. \end{equation*}
\end{theorem}
\begin{proof} The proof of the theorem can be  executed in two parts. \\
\textit{Case 1 -- when $q$ is integer:
} 
We have the following representation from 
\cite{kumar2019tempered}
 \begin{equation*}\mathbb{E}({M_{\alpha,\beta_1,\lambda_1,\mu}(t) })^{q}\sim \left ( \frac{\alpha \beta_1\mu ^{\alpha -1}}{\lambda_1}  \right )^qt^q \;\text{as} \;  t \to \infty. \end{equation*}
\textit{Case 2 -- when $q$ is non-integer: } Assume $0<q<1$. With the help of \cite[eq. (6)]{kumar2019tempered}, we have
\begin{align*}\mathbb{E}\left [ ({M_{\alpha,\beta_1,\lambda_1,\mu}(t) })^{q}\right ] =& \frac{-1}{\Gamma (1-q)}\int_{0}^{\infty }\frac{\mathrm{d} }{\mathrm{d} u}\left [ \frac{\lambda_1}{\lambda_1-\mu^{\alpha}+(\mu+u)^\alpha} \right ]^{\beta_1t}u^{-q} du\\
&=\frac{\alpha \beta_1 t\lambda_1 ^{\beta_1 t}}{\Gamma (1-q)}\int_{0}^{\infty}\frac{(u+\mu)^{\alpha-1}u^{-q}}{\left [  \lambda_1-\mu^{\alpha}+(\mu+u)^\alpha\right ]^{\beta_1 t+1}}du\\ 
&=\frac{\alpha \beta_1 t\lambda_1 ^{\beta_1 t}}{\Gamma (1-q)}\int_{0}^{\infty}\frac{(u+\mu)^{\alpha-1}u^{-q}}{\left [  \lambda_1-\mu^{\alpha}+(\mu+u)^\alpha\right ]}e^{-\beta_1t\ln[\lambda_1-\mu^{\alpha}+(u+\mu)^{\alpha}]}du\\
&=\frac{\alpha \beta_1 t\lambda_1 ^{\beta_1 t}}{\Gamma (1-q)}\int_{\mu}^{\infty}\frac{z^{\alpha-1}(z-\mu)^{-q}}{\left [  \lambda_1-\mu^{\alpha}+z^\alpha\right ]}e^{-\beta_1t\ln[\lambda_1-\mu^{\alpha}+z^{\alpha}]}dz~~~~(\text{by~ letting}~ u+\mu= z).
\end{align*}
By taking $f(z)=\beta_1\ln(\lambda_1-\mu^{\alpha_1}+z^{\alpha})$ and $g(z)=\frac{z^{\alpha-1}(z-\mu)^{-q}}{\left [  \lambda_1-\mu^{\alpha}+z^\alpha\right ]}$, we have the Taylor series around $\mu$ of the form
\begin{align*}f(z)&=\beta_1\ln\lambda_1 + \frac{\beta_1\alpha\mu^{\alpha-1}}{\lambda_1}(z-\mu)+\frac{\beta_1\alpha\mu^{\alpha-2}}{\lambda_1^2}[(\alpha-1)\lambda_1-\alpha\mu^\alpha](z-\mu)^2+\cdots\\
	&=f(\mu)+\sum_{k=0}^{\infty }a_j(z-\mu)^{k+\delta},
\end{align*}
 where $\delta=1,~ f(\mu)=\beta_1\ln\lambda_1, ~a_0=\frac{\beta_1\alpha\mu^{\alpha-1}}{\lambda_1}$, and $a_1=\frac{\beta_1\alpha\mu^{\alpha-2}}{\lambda_1^2}[(\alpha-1)\lambda_1-\alpha\mu^{\alpha}]$. Additionally
\begin{align*}g(z)&= (z-\mu)^{-q}\left[ \frac{\mu^{\alpha-1}}{\lambda_{1}}+\frac{\left( \alpha-1 \right)\lambda_{1}\mu^{\alpha-2}-\alpha\mu^{2\alpha-2}}{\lambda_{1}^{2}}(z-\mu)+ \frac{\lambda_{1}^{2}(\alpha-1)(\alpha-2)\mu^{\alpha-3}-3\alpha(\alpha-1)\mu^{2\alpha-3}\lambda_{1}+2\alpha^{2}\mu^{2\alpha-3}}{\lambda_{1}^{3}}(z-\mu)^2+\cdots \right]\\
&=\sum_{k=0}^{\infty }b_{k}\left( z-\mu \right)^{k+\gamma-1},
\end{align*}
 where $\gamma=1-q,~ b_0=\frac{\mu^{\alpha-1}}{\lambda_1},$ and $b_1=\frac{(\alpha-1)\lambda_1\mu^{\alpha-2}-\alpha\mu^{2\alpha-2}}{\lambda_1^2}$.\newline
Using Laplace-Erdelyi theorem (see \cite[Appendix A]{kumar2019tempered}), we have that
\begin{align*}\int_{\mu}^{\infty }\frac{z^{\alpha-1}(z-\mu)^{-q}}{\lambda_{1}-\mu^{\alpha}+z^{\alpha}}e^{-\beta_1 t\ln(\lambda_{1}-\mu^{\alpha}+z^{\alpha})}dz&\sim e^{-t\beta_1\ln\lambda_{1}}\sum_{k=0}^{\infty }\Gamma\left(\frac{k+\gamma}{\delta}\right)\frac{c_k}{t^{\frac{k+\gamma}{\delta}}}
=\lambda_{1}^{-\beta_1 t}\sum_{k=0}^{\infty }\Gamma(k+1-q)\frac{c_k}{t^{k+1-q}}.
\end{align*}
Hence, we obtain
\begin{align}\label{LT1}\mathbb{E}\left [ ({M_{\alpha,\beta_1,\lambda_1,\mu}(t) })^{q}\right]\sim \frac{\alpha\beta_1 t}{\Gamma(1-q)}\sum_{k=0}^{\infty }\Gamma(k+1-q)\frac{c_k}{t^{k+1-q}},
\end{align}
where $c_k$ in terms of coefficients $a_k$ and $b_k$ is given by \newline
\begin{equation*}c_k=\frac{1}{a_0^{\frac{k+\gamma}{\delta}}}\sum_{j=0}^{k}b_{k-j}\sum_{i=0}^{j}\binom{-\frac{k+\gamma}{\delta}}{i}\frac{1}{a_0^{i}}\hat{B}_{(j,i)}(a_1,a_2,..,a_{j-i+1})
\end{equation*}
and $\hat{B}_{(j,i)}$ are the partial ordinary Bell polynomials (see \cite{andrews1998theory} and \cite{soni2023probabilistic}). 
The dominating term in (\ref{LT1}), for large $t$, leads to 
\begin{align*}\mathbb{E}\left [ ({M_{\alpha,\beta_1,\lambda_1,\mu}(t) })^{q}\right]\sim c_0\alpha \beta_1 t ^q,
\end{align*} 
where $c_0=\left( \frac{\mu^{\alpha-1}}{\lambda_1}\right)^{q}\frac{1}{\left( \alpha\beta_1 \right)^{1-q}}$. Hence 
\begin{align*}\mathbb{E}({M_{\alpha,\beta_1,\lambda_1,\mu}(t) })^{q}\sim \left ( \frac{\alpha \beta_1\mu ^{\alpha -1}}{\lambda_1}  \right )^qt^q\;\; \text{as}\; t \to \infty\; \text{for}\; q\in(0,1).
\end{align*}
On similar lines, we obtain for general $q\in(n-1,n)$
\begin{align*} 
	\mathbb{E}\left [ ({M_{\alpha,\beta_1,\lambda_1,\mu}(t) })^{q}\right ] &=  \frac{(-1)^n}{\Gamma (n-q)}\int_{0}^{\infty }\frac{\mathrm{d}^n }{\mathrm{d} u^n}\left [ \frac{\lambda_1}{\lambda_1-\mu^{\alpha}+(\mu+u)^\alpha} \right ]^{\beta_1t}u^{n-q-1} du \\ 
	& =\frac{(\alpha\beta_1 t)^n}{\Gamma (n-q)}\int_{0}^{\infty } \frac{{\lambda_1^{\beta_1 t} (u+\mu)^{n(\alpha-1)}}}{\left [  \lambda_1-\mu^{\alpha}+(\mu+u)^\alpha\right ]^{\beta_1 t+n}}u^{n-q-1}du \\
	&\sim\frac{(\alpha\beta_1 t)^{n}\lambda_{1}^{\beta_1 t}}{\Gamma(n-q)}\int_{0}^{\infty }\frac{(u+\mu)^{n(\alpha-1)}u^{n-q-1}}{\left[ \lambda_1-\mu^{\alpha}+(\mu+u)^{\alpha} \right]^{n}}e^{-\beta_1 t\ln(\lambda_1-\mu^{\alpha}+(\mu+u)^{\alpha})}du\\
 &=\frac{(\alpha\beta_1 t)^{n}\lambda_{1}^{\beta_1 t}}{\Gamma(n-q)}\int_{\mu}^{\infty }\frac{z^{n(\alpha-1)}(z-\mu)^{n-q-1}}{\left[ \lambda_1-\mu^{\alpha}+z^{\alpha} \right]^{n}}e^{-\beta_1 t\ln(\lambda_1-\mu^{\alpha}+z^{\alpha})}dz~~~(\text{by~ letting}~ u+\mu= z).
\end{align*}
Let $f(z)=\beta_1\ln(\lambda_1-\mu^{\alpha}+z^{\alpha})$ and $g(z)=(z-\mu)^{n-q-1}\left[ \frac{z^{\alpha-1}}{\lambda_1-\mu^{\alpha}+z^{\alpha}} \right]^n$. Then, we have
\begin{align*}
	g(z)& =(z-\mu)^{n-q-1}\left[ \left(\frac{\mu^{\alpha-1}}{\lambda_1} \right)^n+n\left( \frac{\mu^{\alpha-1}}{\lambda_1} \right)^{n-1}\frac{(\alpha-1)\lambda_1\mu^{\alpha-2}-\alpha\mu^{2\alpha-2}}{\lambda_1^2}(z-\mu)+\cdots  \right]\\
&= \sum_{k=0}^{\infty }b_k(z-\mu)^{k+\gamma-1},
\end{align*} 
where $\gamma=n-q,~ b_0=\left(\frac{\mu^{\alpha-1}}{\lambda_1} \right)^n,$ and $ b_1=n\left( \frac{\mu^{\alpha-1}}{\lambda_1} \right)^{n-1}\frac{(\alpha-1)\lambda_1\mu^{\alpha-2}-\alpha\mu^{2\alpha-2}}{\lambda_1^2}$.
With the help of  Laplace-Erdelyi theorem (see \cite[Appendix A]{kumar2019tempered}), we get
\begin{align*}
	\int_{\mu}^{\infty }\frac{z^{n(\alpha-1)}(z-\mu)^{n-q-1}}{\left[ \lambda_1-\mu^{\alpha}+z^{\alpha} \right]^{n}}e^{-\beta_1 t\ln(\lambda_1-\mu^{\alpha}+z^{\alpha})}dz\sim e^{-t\beta_1\ln\lambda_1}\sum_{k=0}^{\infty }\Gamma(k+n-q)\frac{d_k}{t^{k+n-q}}.
 \end{align*}
Hence, we have
\begin{align*}\mathbb{E}\left [ ({M_{\alpha,\beta_1,\lambda_1,\mu}(t) })^{q}\right ]\sim \frac{(\alpha\beta_1 t)^n}{\Gamma(n-q)}\sum_{k=0}^{\infty}\Gamma(k+n-q)\frac{d_k}{t^{k+n-q}},
\end{align*}
where $d_k$ in terms of coefficient of $a_k $ and $b_k$ is given by 
\begin{align*}d_k=\frac{1}{a_0^{\frac{k+\gamma}{\delta}}}\sum_{j=0}^{k}b_{k-j}\sum_{i=0}^{j}\binom{-\frac{k+\gamma}{\delta}}{i}\frac{1}{a_0^{i}}\hat{B}_{(j,i)}(a_1,a_2,..,a_{j-i+1}).
\end{align*}
The dominating term, for large $t$, in the above series  corresponds to 
\begin{align*}\mathbb{E}\left [ ({M_{\alpha,\beta_1,\lambda_1,\mu}(t) })^{q}\right]\sim d_0\frac{(\alpha \beta_1 t )^n}{t^{n-q}},
\end{align*}
where $d_0=\left( \frac{\lambda_1}{\mu^{\alpha-1}}\right)^{-q}\frac{1}{\left( \alpha\beta_1 \right)^{n-q}}$. Therefore
\begin{align*}
	\mathbb{E}({M_{\alpha,\beta_1,\lambda_1,\mu}(t) })^{q}\sim \left ( \frac{\alpha \beta_1\mu ^{\alpha -1}}{\lambda_1}  \right )^qt^q\;\; \text{as}\; t \to \infty \; \text{for}\; q\in(n-1,n).\end{align*}\qedhere
\end{proof}
Next, we present the mean, variance, and autocovariance functions for the TSTFNBP. 
\subsection{Mean, variance, autocovariance and index of dispersion}
\begin{theorem}\label{thm1} Let $0<s\leq t <\infty$, the mean, variance, and autocovariance of the process $\left \{ \mathcal{Q}_{\alpha ,\beta _1,\mu }^{\lambda _1,\beta }(t,\lambda)\right \}_{t\geq0}$ are given by 
	\noindent (i) $\mathbb{E}\left [ \mathcal{Q}_{\alpha ,\beta _1,\mu }^{\lambda _1,\beta }(t,\lambda) \right ]=q_1\mathbb{E}\left [ ({M_{\alpha,\beta_1,\lambda_1,\mu}(t) })^{\beta}  \right ]\sim q_1 \left ( \frac{\alpha \beta_1\mu ^{\alpha -1}t}{\lambda_1}  \right )^{\beta}$\\
   \noindent	(ii) $\text{Var}\left [ \mathcal{Q}_{\alpha ,\beta _1,\mu }^{\lambda _1,\beta }(t,\lambda) \right ]=q_1\mathbb{E}\left [ ({M_{\alpha,\beta_1,\lambda_1,\mu}(t) })^{\beta}  \right ]-q_1^2\left [\mathbb{E}\left [ ({M_{\alpha,\beta_1,\lambda_1,\mu}(t) })^{\beta}  \right ]\right]^2+2c_2\mathbb{E}\left [ ({M_{\alpha,\beta_1,\lambda_1,\mu}(t) })^{2\beta}  \right ]$
  \noindent \begin{align*} \text{(iii)}~  \text{Cov}\left [ \mathcal{Q}_{\alpha ,\beta _1,\mu }^{\lambda _1,\beta }(s,\lambda) ,  \mathcal{Q}_{\alpha ,\beta _1,\mu }^{\lambda _1,\beta }(t,\lambda) \right ]&= q_1\mathbb{E}\left [ ({M_{\alpha,\beta_1,\lambda_1,\mu}(s) })^{\beta}  \right ]+c_1\mathbb{E}\left [ ({M_{\alpha,\beta_1,\lambda_1,\mu}(s) })^{2\beta}  \right ]-q_1^2\mathbb{E}\left [ ({M_{\alpha,\beta_1,\lambda_1,\mu}(s) })^{\beta}  \right ]\mathbb{E}\left [ ({M_{\alpha,\beta_1,\lambda_1,\mu}(t) })^{\beta}  \right ]\\ &+q_1^2\beta\mathbb{E}\left [({M_{\alpha,\beta_1,\lambda_1,\mu}(t) })^{2\beta} B\left(\beta,1+\beta;\frac{M_{\alpha,\beta_1,\lambda_1,\mu}(s)}{M_{\alpha,\beta_1,\lambda_1,\mu}(t)}\right) \right ],
  \end{align*}
	where $q_1=\frac{\lambda}{\Gamma(1+\beta)}, ~c_1=\beta q_1^2B(\beta,1+\beta), ~c_2=\frac{\lambda^2}{\Gamma(2\beta+1)}$, and $B(m,n;x)=\int_{0}^{x}t^{m-1}(1-t)^{n-1}dt$ for $0<x<1$ is an incomplete beta function. 
\end{theorem}
\begin{proof}
	We know that (see \cite{laskin2003fractional})
	\begin{equation}\label{ev1}
	\mathbb{E}\left[N_\beta(t, \lambda)\right] = q_1 t^\beta, \;\;\;\;\text{Var}\left [ N_{\beta}(t,\lambda) \right ]= q_1t^{\beta}\left [ 1+q_1t^{\beta}\left ( \frac{\beta B(\beta,1/2)}{2^{2\beta -1}}-1 \right ) \right ].
	\end{equation}
	Also, from \cite{beghin2009fractional}, we have  that
 \begin{equation*}
	\text{Cov}\left[N_\beta(s, \lambda), N_\beta(t, \lambda)\right] = q_1s^\beta +c_1s^{2\beta}+q_1^2\left[\beta t^{2\beta} B(\beta, 1+\beta;s/t)- (st)^{\beta}\right],\;\; 0< s \leq t,
	\end{equation*}
	\begin{equation*}
	\mathbb{E}\left[N_\beta(s, \lambda)N_\beta(t, \lambda)\right] = q_1s^\beta +c_1s^{2\beta}+ q_1^2\left[\beta t^{2\beta} B(\beta, 1+\beta;s/t)\right],
	\end{equation*}
 and
	\begin{equation*}
	\text{Var}\left [ N_{\beta}(t,\lambda) \right ]= q_1t^{\beta}+\frac{(\lambda t^{\beta })^2}{\beta}\left ( \frac{1}{\Gamma (2\beta )}-\frac{1}{\beta \Gamma^2 (\beta)} \right ).
    \end{equation*}
	Using the conditioning argument and with the help of the above quantities for the FPP, we get
    \begin{align*}
   \mathbb{E}\left [ \mathcal{Q}_{\alpha ,\beta _1,\mu }^{\lambda _1,\beta }(t,\lambda) \right ]=&\mathbb{E}\left [ N_{\beta}(M_{\alpha,\beta_1,\lambda_1,\mu}(t)) \right ]=\mathbb{E}\left [ N_{\beta}(M_{\alpha,\beta_1,\lambda_1,\mu}(t))|M_{\alpha,\beta_1,\lambda_1,\mu}(t) \right ]=q_1\mathbb{E}\left [ ({M_{\alpha,\beta_1,\lambda_1,\mu}(t) })^{\beta}  \right ]\\
     	\text{Var}\left [ \mathcal{Q}_{\alpha ,\beta _1,\mu }^{\lambda _1,\beta }(t,\lambda) \right ]=&\text{Var}\left[ \mathbb{E}\left [ N_{\beta}(M_{\alpha,\beta_1,\lambda_1,\mu}(t))|M_{\alpha,\beta_1,\lambda_1,\mu}(t) \right ] \right] +\mathbb{E}\left [\text{Var}\left[N_{\beta}(M_{\alpha,\beta_1,\lambda_1,\mu}(t))|M_{\alpha,\beta_1,\lambda_1,\mu}(t) \right ]\right]\\
     	=&q_1\mathbb{E}\left [ ({M_{\alpha,\beta_1,\lambda_1,\mu}(t) })^{\beta}  \right ]-q_1^2\left [\mathbb{E}\left [ ({M_{\alpha,\beta_1,\lambda_1,\mu}(t) })^{\beta}  \right ]\right]^2+2c_2\mathbb{E}\left [ ({M_{\alpha,\beta_1,\lambda_1,\mu}(t) })^{2\beta}  \right ]
     	\end{align*}
     and	
\	\begin{align*}
     \mathbb{E} \left [ \mathcal{Q}_{\alpha ,\beta _1,\mu }^{\lambda _1,\beta }(s,\lambda)  \mathcal{Q}_{\alpha ,\beta _1,\mu }^{\lambda _1,\beta }(t,\lambda) \right ]
	&=\mathbb{E} \left [\mathbb{E}\left [N_{\beta}(M_{\alpha,\beta_1,\lambda_1,\mu}(s))N_{\beta}(M_{\alpha,\beta_1,\lambda_1,\mu}(t))|M_{\alpha,\beta_1,\lambda_1,\mu}(s),M_{\alpha,\beta_1,\lambda_1,\mu}(t)\right]\right] \\
	&=\mathbb{E}\left[ q_1({M_{\alpha,\beta_1,\lambda_1,\mu}(s) })^{\beta}+c_1({M_{\alpha,\beta_1,\lambda_1,\mu}(s) })^{2\beta}+q_1^2\beta({M_{\alpha,\beta_1,\lambda_1,\mu}(t) })^{2\beta} B\left(\beta,1+\beta;\frac{M_{\alpha,\beta_1,\lambda_1,\mu}(s)}{M_{\alpha,\beta_1,\lambda_1,\mu}(t)}\right)\right].
	\end{align*}
    Therefore, with the help of Part (i) and Part (ii),  we get the desired expression for the Cov$\left [ \mathcal{Q}_{\alpha ,\beta _1,\mu }^{\lambda _1,\beta }(s,\lambda) ,  \mathcal{Q}_{\alpha ,\beta _1,\mu }^{\lambda _1,\beta }(t,\lambda) \right ]$.
    \end{proof}
\begin{remark}
\noindent Let $\{X(t)\}_{t \geq 0}$ be a stochastic process. We say it is overdispersed if Var$[X(t)]-\mathbb{E}[X(t)]>0$ for all $t\geq 0$ (see \cite[p. 72]{cox1966statistical}).
Now, for the process $\left \{ \mathcal{Q}_{\alpha ,\beta _1,\mu }^{\lambda _1,\beta }(t,\lambda)\right \}_{t\geq0}$
\begin{align*}
\text{Var}\left[\mathcal{Q}_{\alpha ,\beta _1,\mu }^{\lambda _1,\beta }(t,\lambda)\right]-\mathbb{E}\left[\mathcal{Q}_{\alpha ,\beta _1,\mu }^{\lambda _1,\beta }(t,\lambda)\right]
&=2c_2\mathbb{E}\left [ ({M_{\alpha,\beta_1,\lambda_1,\mu}(t) })^{2\beta}  \right ]-q_1^2\left [\mathbb{E}\left [ ({M_{\alpha,\beta_1,\lambda_1,\mu}(t) })^{\beta}  \right ]\right]^2\\
&=\frac{\lambda^2}{\beta}\left[\frac{\mathbb{E}\left [ ({M_{\alpha,\beta_1,\lambda_1,\mu}(t) })^{2\beta}  \right ]}{\Gamma(2\beta)}-\frac{\left [\mathbb{E}\left [ ({M_{\alpha,\beta_1,\lambda_1,\mu}(t) })^{\beta}  \right ]\right]^2}{\beta\Gamma^2 (\beta)}\right]\\
&\geq\left(\mathbb{E}\left [ ({M_{\alpha,\beta_1,\lambda_1,\mu}(t) })^{\beta}  \right ]\right)^2\left[\frac{\lambda^2}{\beta}\left(\frac{1}{\Gamma(2\beta)}-\frac{1}{\beta\Gamma^2 (\beta)}\right)\right].
\end{align*}
One can observe that $\frac{\lambda^2}{\beta} \left(\frac{1}{\Gamma(2 \beta)} - \frac{1}{\beta \Gamma^2 (\beta)}\right) > 0$ for $\lambda >0$ and $\beta \in (0,1)$ (see \cite{beghin2014fractional}) and ${\left [\mathbb{E}\left [ ({M_{\alpha,\beta_1,\lambda_1,\mu}(t) })^{\beta}  \right ]\right]^2}\leq{\mathbb{E}\left [ ({M_{\alpha,\beta_1,\lambda_1,\mu}(t) })^{2\beta}  \right ]}$ is true because of Cauchy-Schwarz inequality and 
 shows the overdispersion of the $\left \{ \mathcal{Q}_{\alpha ,\beta _1,\mu }^{\lambda _1,\beta }(t,\lambda)\right \}_{t\geq0}$.
\end{remark}
\subsection{Laplace transform}Let $h(x,t)$ be the pdf of the $\mathcal{E}_\beta\left(M_{\alpha,\beta_1,\lambda_1,\mu}(t)\right)$ and $k_\beta(x,t)$ be the pdf of the inverse stable subordinator $\mathcal{E}_\beta(t)$ with LT $\mathbb{E}[e^{-u\mathcal{E}_\beta(t)}] =E_{\beta,1}^{1}(-ut^\beta)$ (see \cite{meerschaert2013inverse}). Then, the LT of    $\mathcal{E}_\beta\left(M_{\alpha,\beta_1,\lambda_1,\mu}(t)\right)$ can be derived as
\begin{align}\label{eb1}
\mathbb{E}\left[e^{-u \mathcal{E}_\beta\left(M_{\alpha,\beta_1,\lambda_1,\mu}(t)\right)}\right] &= \int_{0}^{\infty} e^{-ux}h(x,t) dx= \int_{0}^{\infty} \int_{0}^{\infty} e^{-ux} k_\beta(x,y) f_{M_{\alpha,\beta_1,\lambda_1,\mu}}(t,y) dy dx \\
&= \int_{0}^{\infty} E_{\beta,1}^{1}(-uy^\beta) f_{M_{\alpha,\beta_1,\lambda_1,\mu}}(t,y) dy\nonumber\\
	&=  \sum_{l=0}^{\infty}  \frac{(-u)^l}{\Gamma(1+\beta l)}\int_{0}^{\infty} y^{l\beta} f_{M_{\alpha,\beta_1,\lambda_1,\mu}}(t,y)dy\nonumber \\
&= \sum_{l=0}^{\infty}  \frac{(-u)^l}{\Gamma(1+\beta l)} \mathbb{E}[(M_{\alpha,\beta_1,\lambda_1,\mu}(t))^{l\beta}].
\end{align}
Using the conditioning arguments, we obtain the LT for TSTFNBP as
\begin{align*}
\mathbb{E}\left[e^{-u\mathcal{Q}_{\alpha ,\beta _1,\mu }^{\lambda _1,\beta }(t,\lambda)}\right] &=\mathbb{E} \left [\mathbb{E}\left [e^{-uN\left(\mathcal{E}_\beta\left(M_{\alpha,\beta_1,\lambda_1,\mu}(t)\right)\right)}|\mathcal{E}_\beta\left(M_{\alpha,\beta_1,\lambda_1,\mu}(t)\right) \right]\right] \\
 &= \mathbb{E}\left[\mathbb{E}\left[\text{exp}\left(-\lambda \mathcal{E}_\beta\left( M_{\alpha,\beta_1,\lambda_1,\mu}(t)\right)(1-e^{-u})\right)/  \mathcal{E}_\beta\left( M_{\alpha,\beta_1,\lambda_1,\mu}(t)\right)\right]\right]\\
&=  \sum_{l=0}^{\infty}  \frac{(-\lambda(1-e^{-u}))^l}{\Gamma(1+\beta l)} \mathbb{E}[(M_{\alpha,\beta_1,\lambda_1,\mu}(t))^{l\beta}].
\end{align*}
The probability generating function of the process  $\{\mathcal{Q}_{\alpha ,\beta _1,\mu }^{\lambda _1,\beta }(t,\lambda)\}_{t \geq 0}$ can be evaluated using LT and is given by
\begin{equation*}
\mathbb{E}\left[u^{\mathcal{Q}_{\alpha ,\beta _1,\mu }^{\lambda _1,\beta }(t,\lambda)}\right] =  \sum_{l=0}^{\infty}  \frac{(-\lambda(1-u))^l}{\Gamma(1+\beta l)} \mathbb{E}[(M_{\alpha,\beta_1,\lambda_1,\mu}(t))^{l\beta}].
\end{equation*} 
\begin{remark}
\noindent When $\alpha =1,\mu=0$, the LT of TSTFNBP  reduces to  
		\begin{equation*}
		\mathbb{E}\left[e^{-u\mathcal{Q}_{1 ,\beta _1,0 }^{\lambda _1,\beta }(t,\lambda)}\right] =\frac{1}{ \Gamma(\beta_1 t)}\;  _2\psi_1 \left[\frac{-\lambda(1-e^{-u})}{\lambda_1^{\beta}}\; \vline \;\begin{matrix}
\left(1,1\right),& \left(\beta_1 t , \beta\right)\\
(1, \beta)
\end{matrix} \right],
		\end{equation*}
		which is the LT of the FNBP discussed in \cite{samy2018fractional}.
\end{remark}
 Next, we present the L\'{e}vy measure density for the particular case when $\beta =1.$
\subsection{L\'{e}vy measure}
It is to note that when $\beta=1$, the TSTFNBP is identical to  the TSFNBP defined in \cite{maheshwari2023tempered}. We derive here the L\'{e}vy measure density for TSFNBP. 
Using L\'{e}vy density of TMLLP (\ref{ld1}) and the formula  \cite[page 197]{ken1999levy}, the L\'{e}vy measure $\mathcal{D}$ for the process  $\{\mathcal{Q}_{\alpha ,\beta _1,\mu }^{\lambda _1,1 }(t,\lambda)\}_{t\geq 0}$ can be evaluated as
\begin{align}\label{lm56}
\mathcal{D}(k) 
&= \int_{0}^{\infty} \sum_{i=1}^{\infty}p_{\beta}(i/t, \lambda)  \delta_{\{i\}}(k)\pi(t) dt  \nonumber \\
&= \int_{0}^{\infty} \sum_{i=1}^{\infty} \frac{(\lambda t)^i}{i!} e^{-\lambda t} \delta_{\{i\}}(k) \frac{\alpha\beta_1}{t}e^{-\mu t}E_{\alpha, 1}\left[(\mu^{\alpha}-\lambda_1) t^\alpha\right] dt  \nonumber\\
&= \alpha \beta_1\sum_{i=1}^{\infty}  \frac{(\lambda)^i}{i!}\delta_{\{i\}}(k) \int_{0}^{\infty} e^{-(\mu+\lambda) t} t^{i-1}\sum_{j=0}^{\infty }\frac{\left [ (\mu ^{\alpha}-\lambda_1)t^{\alpha }\right ]^{j}}{\Gamma (\alpha j+1)} dt\nonumber \\
&=  \alpha \beta_1\sum_{i=1}^{\infty}  \frac{(\lambda)^i}{i!}\delta_{\{i\}}(k) \sum_{j=0}^{\infty }\frac{(\mu ^{\alpha}-\lambda_1)^{j}}{\Gamma (\alpha j+1)} \int_{0}^{\infty}  e^{-(\mu+\lambda) t}t^{i-1+\alpha j}  dt \nonumber \\
&=  \alpha \beta_1\sum_{i=1}^{\infty}  \frac{(\lambda)^i}{i!}\delta_{\{i\}}(k) \sum_{j=0}^{\infty }\frac{(\mu ^{\alpha}-\lambda_1)^{j}}{\Gamma (\alpha j+1)}\frac{\Gamma(\alpha j+i)}{(\lambda+\mu)^{\alpha j+i}}. 
\end{align}
\section{Second-order asymptotic properties and first-passage time}\label{sec:lrd}
In this section, we will discuss the long-range dependence and first-passage time for the TSTFNBP. We first reproduce the definition of the LRD property (see \cite{d2014time, maheshwari2016long, kumar2020fractional}).
\begin{definition}
	Let $0 <s <t,$ let the correlation function Corr$[X(s), X(t)]$ for a stochastic process $\{X(t)\}_{t \geq 0}$ satisfies the following relation
	\begin{equation*}
	\lim_{t \rightarrow \infty} \frac{ \text{ Corr}[X(s), X(t)]}{t^{-d}} = k(s),
	\end{equation*}
	for some $k(s) >0$ and $d > 0$.
	The process $\{X(t)\}_{t \geq 0}$ exhibits the long-range dependence (LRD) property when $d \in (0,1)$.
\end{definition}
\subsection{Dependence structure of the TSTFNBP}
\begin{lemma}\label{lm22}
	Let $\beta \in (0,1)$ and $0 < s < t, s$ is fixed. Then the following asymptotic expansion holds for a large t.\newline
\noindent (i)\;\; $\mathbb{E}\left[(M_{\alpha,\beta_1,\lambda_1,\mu}(s))^{\beta} (M_{\alpha,\beta_1,\lambda_1,\mu}(t))^{\beta} \right] \sim \mathbb{E}\left[(M_{\alpha,\beta_1,\lambda_1,\mu}(s))^{\beta}\right] \mathbb{E}\left[(M_{\alpha,\beta_1,\lambda_1,\mu}(t-s))^{\beta}\right]$ \\
\noindent (ii)\;\;$\beta \mathbb{E}\left[(M_{\alpha,\beta_1,\lambda_1,\mu}(t))^{2\beta}B\left(\beta, 1+\beta; \frac{M_{\alpha,\beta_1,\lambda_1,\mu}(s)}{M_{\alpha,\beta_1,\lambda_1,\mu}(t)} \right)\right] \sim \mathbb{E}\left[(M_{\alpha,\beta_1,\lambda_1,\mu}(s))^{\beta}\right] \mathbb{E}\left[(M_{\alpha,\beta_1,\lambda_1,\mu}(t-s))^{\beta}\right].$
\end{lemma}
\begin{proof}
    \noindent Proof of the following lemma is similar to that of  Lemma 2 in \cite{maheshwari2016long}.
\end{proof}
\noindent We will next proof the LRD property for our process.
\begin{theorem}
	The process TSTFNBP exhibits the LRD property.
\end{theorem}
\begin{proof} Using Theorem \ref{prop2} and with the help of Lemma \ref{lm22}(ii), the asymptotic behavior of the covariance is 
	\begin{align*}
	  \text{Cov}\left [ \mathcal{Q}_{\alpha ,\beta _1,\mu }^{\lambda _1,\beta }(s,\lambda) ,  \mathcal{Q}_{\alpha ,\beta _1,\mu }^{\lambda _1,\beta }(t,\lambda) \right ] &\sim q_1\mathbb{E}\left [ ({M_{\alpha,\beta_1,\lambda_1,\mu}(s) })^{\beta}  \right ]+c_1\mathbb{E}\left [ ({M_{\alpha,\beta_1,\lambda_1,\mu}(s) })^{2\beta}  \right ]\\ & \;\;\;\;-q_1^2\mathbb{E}\left [ ({M_{\alpha,\beta_1,\lambda_1,\mu}(s) })^{\beta}  \right ]\left [\mathbb{E}\left [ ({M_{\alpha,\beta_1,\lambda_1,\mu}(t) })^{\beta}  \right ] - \mathbb{E}\left [ ({M_{\alpha,\beta_1,\lambda_1,\mu}(t-s) })^{\beta}  \right ]\right ]\\
	 &\sim q_1\mathbb{E}\left [ ({M_{\alpha,\beta_1,\lambda_1,\mu}(s) })^{\beta}  \right ]+c_1\mathbb{E}\left [ ({M_{\alpha,\beta_1,\lambda_1,\mu}(s) })^{2\beta} \right ]\\ &\;\;\;\; -q_1^2\mathbb{E}\left [ ({M_{\alpha,\beta_1,\lambda_1,\mu}(s) })^{\beta}  \right ]\left[ \left( \frac{\alpha\beta_1\mu^{\alpha-1}t}{\lambda_1}\right)^{\beta} -\left(\frac{\alpha\beta_1\mu^{\alpha-1}(t-s)}{\lambda_1}\right)^{\beta}\right ] \\
	 & \sim q_1\mathbb{E}\left [ ({M_{\alpha,\beta_1,\lambda_1,\mu}(s) })^{\beta}  \right ]+c_1\mathbb{E}\left [ ({M_{\alpha,\beta_1,\lambda_1,\mu}(s) })^{2\beta} \right ]\;\;\; (\text{since}\;\; t^{\beta} - (t-s)^{\beta} \sim \beta st^{\beta -1} \text{for large value of}\; t).
		\end{align*}
		Also, the asymptotic behavior of the variance is as follows
		\begin{align*}
		\text{Var}\left [ \mathcal{Q}_{\alpha ,\beta _1,\mu }^{\lambda _1,\beta }(t,\lambda) \right ]
		&\sim q_1\left(\frac{\alpha\beta_1\mu^{\alpha-1}t}{\lambda_1}\right)^{\beta}-q_1^2\left(\frac{\alpha\beta_1\mu^{\alpha-1}t}{\lambda_1}\right)^{2\beta}+2c_2\left(\frac{\alpha\beta_1\mu^{\alpha-1}t}{\lambda_1}\right)^{2\beta}\\
&=t^{2\beta}\left[q_1\left(\frac{\alpha\beta_1\mu^{\alpha-1}}{\lambda_1 t}\right)^{\beta}-q_1^2\left(\frac{\alpha\beta_1\mu^{\alpha-1}}{\lambda_1}\right)^{2\beta}+2c_2\left(\frac{\alpha\beta_1\mu^{\alpha-1}}{\lambda_1}\right)^{2\beta}\right]\\
		&\sim t^{2\beta}\left(\frac{\alpha\beta_1\mu^{\alpha-1}}{\lambda_1}\right)^{2\beta}(2c_2-q_1^2)\\
		&\sim t^{2\beta} d_1,
			\end{align*}
where $d_1=\left(\frac{\alpha\beta_1\mu^{\alpha-1}}{\lambda_1}\right)^{2\beta}(2c_2-q_1^2)$.
Therefore, the correlation function can be computed as
	\begin{align*}
	\text{Corr}\left [ \mathcal{Q}_{\alpha ,\beta _1,\mu }^{\lambda _1,\beta }(s,\lambda) ,  \mathcal{Q}_{\alpha ,\beta _1,\mu }^{\lambda _1,\beta }(t,\lambda) \right ] &= \frac{\text{Cov}\left [ \mathcal{Q}_{\alpha ,\beta _1,\mu }^{\lambda _1,\beta }(s,\lambda) ,  \mathcal{Q}_{\alpha ,\beta _1,\mu }^{\lambda _1,\beta }(t,\lambda) \right ]}{\sqrt{\text{Var}\left [ \mathcal{Q}_{\alpha ,\beta _1,\mu }^{\lambda _1,\beta }(s,\lambda)\right ]}\sqrt{\text{Var}\left [\mathcal{Q}_{\alpha ,\beta _1,\mu }^{\lambda _1,\beta }(t,\lambda) \right ]}} \sim \frac{q_1\mathbb{E}\left [ ({M_{\alpha,\beta_1,\lambda_1,\mu}(s) })^{\beta}  \right ]+c_1\mathbb{E}\left [ ({M_{\alpha,\beta_1,\lambda_1,\mu}(s) })^{2\beta} \right ]}{\sqrt{t^{2\beta}d_1}\sqrt{\text{Var}\left [\mathcal{Q}_{\alpha ,\beta _1,\mu }^{\lambda _1,\beta }(t,\lambda) \right ]} }\\
	&= t^{-\beta}\left[\frac{q_1\mathbb{E}\left [ ({M_{\alpha,\beta_1,\lambda_1,\mu}(s) })^{\beta}  \right ]+c_1\mathbb{E}\left [ ({M_{\alpha,\beta_1,\lambda_1,\mu}(s) })^{2\beta} \right ]}{\sqrt{d_1\text{Var}\left [\mathcal{Q}_{\alpha ,\beta _1,\mu }^{\lambda _1,\beta }(t,\lambda) \right ]} }  \right].
	\end{align*}
	Hence, for $0< \beta < 1$ and the decaying power $t^{-\beta}$, the process shows the LRD property.
\end{proof}
\subsection{First-passage time distribution}
Finally, we look at the first-passage time distribution of the TSTFNBP. For a stochastic process, it is the time during which a process reaches a certain threshold for the first time. 

Let $\mathcal{T}_k$ be the time of first upcrossing of level $k$ and is defined as
$	\mathcal{T}_k := \inf\{t\geq0 :\mathcal{Q}_{\alpha ,\beta _1,\mu }^{\lambda _1,\beta }(t,\lambda)\geq k \}.$ Then the survival function $Pr\{\mathcal{T}_k>t\}$ can be derived as
\begin{align*}
	Pr\{\mathcal{T}_k>t\} &= Pr\{\mathcal{Q}_{\alpha ,\beta _1,\mu }^{\lambda _1,\beta }(t,\lambda)<k \}
	= \sum_{n=0}^{k-1}Pr\{\mathcal{Q}_{\alpha ,\beta _1,\mu }^{\lambda _1,\beta }(t,\lambda)=n \}\\
	&=\sum_{n=0}^{k-1}\int_{0}^{\infty }p_{\beta }({n}/{y},\lambda ) f_{M_{\alpha,\beta_1,\lambda_1,\mu}(t,y)}dy\\
	&=\sum_{n=0}^{k-1}\int_{0}^{\infty }\frac{(\lambda y^{\beta})^n}{n!} \sum_{l=0}^{\infty} \frac{(n+l)!}{l!}\frac{(-\lambda y^{\beta})^l}{\Gamma(\beta(l+n)+1)}f_{M_{\alpha,\beta_1,\lambda_1,\mu}(t,y)}dy\\
	&=\sum_{n=0}^{k-1}\frac{\lambda^n}{n!} \sum_{l=0}^{\infty} \frac{(n+l)!}{l!}\frac{(-\lambda )^l}{\Gamma(\beta(l+n)+1)}\int_{0}^{\infty }y^{\beta(n+l)}f_{M_{\alpha,\beta_1,\lambda_1,\mu}(t,y)}dy\\
	&=\sum_{n=0}^{k-1}\frac{\lambda^n}{n!} \sum_{l=0}^{\infty} \frac{(n+l)!}{l!}\frac{(-\lambda )^l}{\Gamma(\beta(l+n)+1)}\mathbb{E}[(M_{\alpha,\beta_1,\lambda_1,\mu}(t))^{\beta(n+l)}].
\end{align*}
Furthermore, the distribution of $\mathcal{T}_k$ can be written as 
\begin{align*}
	Pr\{\mathcal{T}_k<t\} &= Pr\{\mathcal{Q}_{\alpha ,\beta _1,\mu }^{\lambda _1,\beta }(t,\lambda)\geq k \}	
	=\sum_{n=k}^{\infty} Pr\{\mathcal{Q}_{\alpha ,\beta _1,\mu }^{\lambda _1,\beta }(t,\lambda)=n \}	\\
	&=\sum_{n=k}^{\infty}\frac{\lambda^n}{n!} \sum_{l=0}^{\infty} \frac{(n+l)!}{l!}\frac{(-\lambda )^l}{\Gamma(\beta(l+n)+1)}\mathbb{E}[(M_{\alpha,\beta_1,\lambda_1,\mu}(t))^{\beta(n+l)}].
\end{align*}
Therefore, the density function $\mathcal{P}(n,t)= Pr\{\mathcal{T}_k \in dt \}/dt$ is
\begin{align*}
	\mathcal{P}(n,t)&= \frac{d}{dt}\sum_{n=k}^{\infty}\frac{\lambda^n}{n!} \sum_{l=0}^{\infty} \frac{(n+l)!}{l!}\frac{(-\lambda )^l}{\Gamma(\beta(l+n)+1)}\mathbb{E}[(M_{\alpha,\beta_1,\lambda_1,\mu}(t))^{\beta(n+l)}]\\
	&=\frac{d}{dt}\left ( 1-\sum_{n=0}^{k-1}\frac{\lambda^n}{n!} \sum_{l=0}^{\infty} \frac{(n+l)!}{l!}\frac{(-\lambda )^l}{\Gamma(\beta(l+n)+1)}\mathbb{E}[(M_{\alpha,\beta_1,\lambda_1,\mu}(t))^{\beta(n+l)}] \right )\\
	&= -\frac{d}{dt}\sum_{n=0}^{k-1}\frac{\lambda^n}{n!} \sum_{l=0}^{\infty} \frac{(n+l)!}{l!}\frac{(-\lambda )^l}{\Gamma(\beta(l+n)+1)}\mathbb{E}\left[(M_{\alpha,\beta_1,\lambda_1,\mu}(t))^{\beta(n+l)}\right].
\end{align*}

\end{document}